\documentclass[12pt]{article}
\usepackage{latexsym,amsfonts,amsmath}
\usepackage{amssymb,amsmath,amsthm}
\def\qi#1 {\fbox {\footnote {\ }}\ \footnotetext { From Qi: {\color{red}#1}}}

\newtheorem{theorem}{Theorem}[section]
\newtheorem{definition}[theorem]{Definition}
\newtheorem{remark}[theorem]{Remark}
\newtheorem{lemma}[theorem]{Lemma}
\newtheorem{corollary}[theorem]{Corollary}
\newtheorem{proposition}[theorem]{Proposition}
\newtheorem{fact}{Fact}

\newcommand{\cP}{\mathcal P}
\newcommand{\Z}{\mathbb Z}
\newcommand{\Q}{\mathbb Q}
\newcommand{\N}{\mathbb N}

\newcommand{\C}{\mathbb C}

\newcommand{\lcm}{\mathrm{lcm}}
\newcommand{\ord}{\mathrm{ord}}

\begin{document}

\title{New nonexistence results on $(m,n)$-generalized bent functions}
\author{Ka Hin Leung \footnote{Research is supported by grant
R-146-000-158-112, Ministry of Education, Singapore} \\ Department of Mathematics\\
National University of Singapore\\ Kent Ridge, Singapore 119260\\
Republic of Singapore \\[1cm]
Qi Wang\\ Department of Computer Science and Engineering\\
Southern University of Science and Technology\\
Shenzhen, Guangdong 518055, China
}
\date{}
\maketitle

\begin{abstract}
  In this paper, we present some new nonexistence results on $(m,n)$-generalized bent functions, which improved recent results. More precisely, we derive new nonexistence results for general $n$ and $m$ odd or $m \equiv 2 \pmod{4}$, and further explicitly prove nonexistence of $(m,3)$-generalized bent functions for all integers $m$ odd or $m \equiv 2 \pmod{4}$. The main tools we utilized are certain exponents of minimal vanishing sums from applying characters to group ring equations that characterize $(m,n)$-generalized bent functions.

\end{abstract}

\section{Introduction}\label{sec-intro}

Let $m \ge 2$, $n$ be positive integers, and $\zeta_m = e^{\frac{2\pi \sqrt{-1}}{m}}$ be a primitive complex $m$-th root of unity. A function $f: \Z_2^n \rightarrow \Z_m$ is called an {\em $(m,n)$-generalized bent function} (GBF) if 
\begin{equation}\label{eqn-gbf}
  |F(y)|^2 = 2^n
\end{equation}
for all $y \in \Z_2^n$, where $F(y)$ is defined as 
\begin{equation}\label{eqn-fy}
  F(y) := \sum_{x \in \Z_2^n} \zeta_m^{f(x)} (-1)^{y \cdot x},
\end{equation}
and $y \cdot x$ denotes the usual inner product. In particular, when $m = 2$, the generalized bent functions defined above are simply boolean bent functions introduced by Rothaus~\cite{Rothaus76}, whereas the function $F: \Z_2^n \rightarrow \mathbb{R}$ in fact becomes the Fourier transform of the boolean function $f$. In 1985, Kumar, Scholtz, and Welch~\cite{KSW85} generalized the notion of boolean bent function by considering bent functions from $\Z_m^n$ to
$\Z_m$. For recent nonexistence results on such generalized bend functions, see Leung and Schmidt~\cite{LS19}. Schmidt~\cite{Schmidt09} investigated generalized bent functions from $\Z_2^n$ to $\Z_m$ for their applications in CDMA communications. For the boolean case, it is well known that bent function exists if and only if $n$ is even, and many constructions were reported (for a survey see~\cite{Carlet10}). In the literature, there exist constructions of generalized bent function from $\Z_2^n$ to $\Z_m$ for $m = 4, 8, 2^k$ (for example,
see~\cite{Schmidt09,Schmidt092,SMGS13,TXQF17}). Very recently, Liu, Feng and Feng~\cite{LFF17} presented several nonexistence results on generalized bent functions from $\Z_2^n$ to $\Z_m$. In this paper, we continue to investigate the nonexistence of such generalized bent functions, and present more new nonexistence results. If $m$ and $n$ are both even or $m$ is divisible by $4$, then there exists an $(m,n)$-generalized bent function~\cite{LFF17}. Therefore, we
restrict attention to the following two cases: 
\begin{itemize}
  \item[(i)] $m$ is odd; 
  \item[(ii)] $n$ is odd and $m \equiv 2 \pmod{4}$.
\end{itemize}
In the following, we always assume that $m$ is odd or $m = 2m'$ with $m'$ odd.    

The remainder of this paper is organized as follows. In Section~\ref{sec-pre}, we introduce some basic tools and auxiliary results. In Section~\ref{sec-main}, we give several new nonexistence results of $(m,n)$-generalized bent functions, which improve the recent results in~\cite{LFF17}. Furthermore, we show that no $(m,3)$-GBF exists for all $m$ odd or $m \equiv 2 \pmod{4}$ in Section~\ref{sec-n3}. 
%Section~\ref{sec-con} concludes the paper.  

\section{Basic tools and auxiliary results}\label{sec-pre}

In this section, we introduce some basic tools and auxiliary results, which will be used in later sections.

\subsection{Group ring and character theory}\label{sec-pre1}

It turns out that group ring and characters of abelian groups play an important role in the study of GBFs. Let $G$ be a finite group of order $v$, $R$ be a ring, and $R[G]$ denote the group ring of $G$ over $R$. For a subset $D$ of a group $G$, we may identify $D$ with the group ring element $\sum_{g \in G} d_g g \in R[G]$, also denoted by $D$ by abuse of notation, where $d_g \in R$ and these $d_g$'s are called {\em coefficients} of $D$. Let $1_G$ denote the
identity element of $G$ and let $r$ be an element in $R$. For simplicity, we write $r$ for the group ring element $r 1_G \in R [G]$. For the group ring element $D = \sum_{g \in G} d_g g \in R[G]$, its {\em support} is defined as
$$
supp(D) := \{ g \in G: d_g \ne 0 \},
$$
and we also define $|D| := \sum d_g$ and $|| D || := \sum |d_g|$. Let $t$ be an integer coprime to $m$. For $D = \sum_{g \in G} d_g g \in \Z[\zeta_m] [G]$, we write $D^{(t)} = \sum d_g^\sigma g^t$, where $\sigma$ is the automorphism of $\Q[\zeta_m]$ determined by $\zeta_m^\sigma = \zeta_m^t$.

The group ring notation is very useful when applying characters. A {\em character} $\chi$ of $G$ is a homomorphism $\chi: G \rightarrow \C^*$. The set of all such characters forms a group $\hat{G}$ which is isomorphic to $G$ itself, and the identity element of $\hat{G}$, denoted by $\chi_0$, which maps every element in $G$ to $1$ (i.e., $\chi_0 (g) = 1$ for all $g \in G$), is called the {\em principal character} of $G$. It is clear that the character group has the
multiplication in $\hat{G}$ defined by $\chi\tau (g) = \chi(g) \tau(g)$ for $\chi, \tau \in \hat{G}$. For $D = \sum_{g \in G} d_g g \in \C[G]$ and $\chi \in \hat{G}$, we have $\chi(D) = \sum_{g \in G} d_g \chi (g)$. For a subgroup $U$ of the group $G$, we define a subgroup of $\hat{G}$ as $U^\perp := \{ \chi \in \hat{G}: \chi(g) = 1 \textrm{ for all $g \in U$} \}$. If $\chi \in U^\perp$, we say that the character $\chi$ is {\em trivial on $U$}. It is easy to see that $| U^\perp | = |G|/ |U|$. The following two results are standard and well-known in character theory. 

\begin{fact}[Orthogonality relations]\label{fact1}
Let $G$ be a finite abelian group of order $v$ with identity $1_G$. Then
$$
\sum_{\chi \in \hat{G}} \chi (g) = \left\{\begin{array}{cc}
  0 & \textrm{ if $g \ne 1_G$},\\
  v & \textrm{ if $g = 1_G$}, 
\end{array}\right.
$$
and 
$$
\sum_{g \in G} \chi(g) = \left\{\begin{array}{cc}
  0 & \textrm{ if $ \chi \ne \chi_0$}, \\
  v & \textrm{ if $\chi = \chi_0$}.
\end{array}\right.
$$
\end{fact}

\begin{fact}[Fourier inversion formula]\label{fact2}
  Let $G$ be a finite abelian group of order $v$, let $D = \sum_{g \in G} d_g g \in \C[G]$ by abuse of notation and $\chi(D) = \sum_{g \in G} d_g \chi(g)$. Then the coefficients in $D$ are determined by 
  $$
  d_g  = \frac{1}{v} \sum_{\chi \in \hat{G}} \chi(D g^{-1}).
  $$
\end{fact}

\subsection{Some auxiliary results}\label{sec-pre2}

We now characterize $(m,n)$-generalized bent functions using the group ring equations. Instead of working with additive groups, we use multiplicative notation. We denote the cyclic group of order $m$ by $C_m$, and set $G = C_2^n$. Whenever $s | m$, we also denote the subgroup of order $s$ in $C_m$ by $C_s$.  

%It seems more convenient to interpret $\Z_2^n$ as $C_2^n$, where $C_m$ denotes the cyclic group of order $m$. Hereafter, assume that $G = C_2^n$.

\begin{definition}\label{def-bfdf}
Let $f: G \rightarrow \Z_m$ be a function, and $g$ be a generator of $C_m$. We define an element $B_f$ in the group ring $\Z[\zeta_m] [G]$ corresponding to $f$ by 
$$
B_f := \sum_{x \in G} \zeta_m^{f(x)} x .
$$
Furthermore, we define an element $D_f$ in the group ring $\Z[C_m] [G]$ by
$$
D_f := \sum_{x \in G} g^{f(x)} x.  
$$
\end{definition}

\begin{remark}\label{rmk-not}
To study $(m,n)$-GBFs, we may assume that $C_m=\langle \{ g^{f(x)}: x\in G\}\rangle$. By scaling if necessary, we may always assume $f(1_G)=0$, i.e., $g^{f(1_G)} = g^0$ is the identity element of $C_m$. From time to time, we may also interpret $\Z[C_m][G]$ as $\Z[C_m \cdot G]$, where $g^0$ and $1_G$ in $\Z[C_m \cdot G]$ both denote the identity element of $C_m \cdot G$. 
\end{remark}

Let $\tau$ be a character that maps $g$ to $\zeta_m$, then it is clear that $\tau(D_f) = B_f$. Moreover, every element $y \in G$ determines a character $\chi_y$ of $G$ by 
$$
\chi_y (x) = (-1)^{y \cdot x},
$$
for all $x \in G$. It is easily verified that every complex character of $G$ is equal to some $\chi_y$ with $y \in G$. Note that
\begin{equation}\label{eqn-chibf}
  \chi_y (B_f) = \sum_{x \in G} \zeta_m^{f(x)} \chi_y(x) = \sum_{x \in G} \zeta_m^{f(x)} (-1)^{y \cdot x} = F(y),
\end{equation}
for all $y \in G$, where $F(y)$ is defined in (\ref{eqn-fy}). It then follows from (\ref{eqn-gbf}) and (\ref{eqn-chibf}) that $f$ is an $(m,n)$-GBF if and only if 
\begin{equation}\label{eqn-chibf2}
  |\chi(B_f)|^2 = 2^n,
\end{equation}
for all $\chi \in \hat{G}$. We now have the following characterization of $(m,n)$-GBFs. 

\begin{proposition}\label{prop-gbf}
Let $f$ be a function from $G$ to $\Z_m$. Then $f$ is an $(m,n)$-GBF if and only if 
\begin{equation}\label{eqn-gbfbf}
  B_f B_f^{(-1)} = 2^n.
\end{equation}
Furthermore, if $f(G) = 2 \Z_m$, then $f$ can be regarded as an $(m',n)$-GBF, where $m = 2m'$ with $m'$ odd.
\end{proposition}

\begin{proof}
  From (\ref{eqn-chibf2}) it follows that 
  $$
  |\chi(B_f)|^2 = \chi (B_f B_f^{(-1)}) = 2^n,
  $$
  for all characters $\chi$ of $G$. Using Facts~\ref{fact1} and~\ref{fact2}, we are able to determine all the coefficients of $B_f B_f^{(-1)}$, i.e., (\ref{eqn-chibf2}) holds if and only if (\ref{eqn-gbfbf}) is satisfied. The last statement follows from the fact that $\zeta_m^{f(x)}$ becomes an $m'$-th root of unity.  
\end{proof}

Observe that we may write 
\begin{equation}\label{eqn-dfex}
  D_f D_f^{(-1)} = \sum_{x \in G} \sum_{y \in G} g^{f(y + x)} g^{-f(y)} x = \sum_{x \in G} E_x x, 
\end{equation}
where $E_x = \sum_{y \in G} g^{f(y + x)} g^{-f(y)} \in \Z[C_m]$.

\begin{lemma}\label{lem-ex}
Suppose that $f$ is a GBF from $G$ to $\Z_m$. Then 
\begin{itemize}
\item [(a)] $E_x=E_x^{(-1)}$ and the coefficient of $g^0$ in $E_x$ is even for all $x\in G$; 
\item [(b)] For each character $\tau$ of order $m$ on $C_m$, we have $\tau(E_x) = 0$ for all $x \ne 1_G$.
\end{itemize}
\end{lemma}
\begin{proof}
Note that $(D_f D_f^{(-1)})^{(-1)}=D_fD_f^{(-1)}$. Hence, we have $E_x=E_x^{(-1)}$ for all $x\in G$. Note that $E_{1_G}=2^n$. Thus, we may consider $x\neq 1_G$. Suppose that $x\neq 1_G$ and $(g_1x_1)(g_2x_2)^{-1}=g^{0}x$ for some $g_1,g_2\in C_m$ and $x_1,x_2\in G$. Note that $x_1\neq x_2$ and clearly, we have $(g_2x_2)(g_1x_1)^{-1}=g^{0}x$ as well. This shows that the coefficient of $g^0$ in $E_x$ is even. 

For any  character $\tau$ of order $m$ on $C_m$, we obtain
$$
\tau(D_f) \tau(D_f)^{(-1)} = B_f B_f^{(-1)} = 2^n = \sum_{x \in G} \tau(E_x) x.
$$
From (\ref{eqn-gbfbf}) in Proposition~\ref{prop-gbf}, the conclusion follows.
\end{proof}

The key in our study of $(m,n)$-GBFs is to investigate $E_x$. Lemma~\ref{lem-ex} (b) allows us to define the notion of vanishing sum (v-sum), which was also studied in details in~\cite{LL00}. Another important notion to study v-sum is the idea of exponents and reduced exponents defined in~\cite{Lenstra79}. In Section~\ref{sec-main}, we will use exponents to derive some new nonexistence results. To this end, we recall some notations defined in~\cite{Lenstra79} and prove some preliminary lemmas. 

Let $S$ be a finite index set, and we denote by $\cP(k)$ the set of all prime factors of the integer $k$. 

\begin{definition}\label{def-exp}
Suppose that $ X = \sum_{i\in S} a_i \mu_i $ where $\mu_i$'s are distinct roots of unity and all $a_i$'s are nonzero integers. We say that $u$ is the {\em exponent} of $X $ if $u$ is the smallest positive integer such that $\mu_i^u = 1$ for all $i$. We say that $k$ is the {\em reduced exponent} of $X $ if $k$ is the smallest positive integer such that there exists $j$ with $(\mu_i \mu_j^{-1} )^k = 1$ for all $i$.  
\end{definition}

For example, the exponent of $\sum_{i=0}^{p-1} \zeta_3 \zeta_p^i$ is $3p$, whereas the reduced exponent is $p$. To study vanishing sums, we consider those which are minimal.  

\begin{definition}\label{def-mini}
Suppose that $ X = \sum_{i \in S} a_i \mu_i=0 $ where $\mu_i$'s are distinct roots of unity and all $a_i$'s are nonzero integers. We say that the relation $X = 0$ is {\em minimal}, if for any proper subset $I \subsetneq S$, $\sum_{i \in I} a_i \mu_i \ne 0$.  
\end{definition}

Based on the definition of minimal relation, we have the following restriction on the cardinality of the index set $S$, in terms of the reduced exponents of a minimal vanishing sum.

\begin{proposition}~\cite{CJ76}\label{prop-rr}
  Suppose that $ X = \sum_{i \in S} a_i \mu_i = 0$ is a minimal relation with reduced exponent $k$ and all $a_i$'s are nonzero. Then $k$ is square free and 
  $$
  |S| \geq 2 +  \sum_{p \in \cP(k)} (p - 2).
  $$
\end{proposition}

For convenience, we define the following notation. 

\begin{definition}
For any group $H$, by $\N[H]$ we denote 
\[ 
\left\{ \sum_{g\in H} a_g g: a_g\in \Z \mbox{ and } a_g\geq 0 \right\} .
\]
\end{definition}

Now we consider the corresponding notion of minimal relation in $\N[C_m]$. From now on, we assume that $g$ is a generator of $C_m$. We recall the notion of minimality defined in Section 4 of~\cite{LL00}. 

\begin{definition}~\cite{LL00}\label{def-nmini}
  Let $D=\sum_{i =0 }^{m-1} a_i g^i\in \N[C_m]$. We say that $D$ is a {\em v-sum} if there exists a character $\tau$ of order $m$ such that $\tau(D) = \tau(\sum_{i =0}^{m-1} a_i g^i) = 0$. We say that $D$ is {\em minimal} if $\tau(\sum_{i = 0}^{m-1} b_i g^i)\neq 0$ whenever $0\leq b_i \leq a_i$ for all $i$ and $b_j < a_j$ for some $j$.
\end{definition}

%Note that the notion of n-minimality defined above is equivalent to that of minimality in~\cite{LL00}. It is clear that if $D$ is n-minimal, then $D$ is also minimal. To study those group ring elements which are not $n$-minimal, we first split them into sum of minimal elements. We first record the following obvious result.

Suppose that $S \subseteq \{0, \ldots, m-1\}$ and $a_i > 0$ for all $i \in S$. It is clear that if $D = \sum_{i \in S} a_i g^i$ is a minimal v-sum by Definition~\ref{def-nmini}, then $\tau(D) = \sum_{i \in S} a_i \tau(g)^i$ is a minimal relation by Definition~\ref{def-mini}. We now define the {\em reduced exponent} of $D$ as follows.

\begin{definition}\label{def-grexp}
  %Let $\tau$ be a character of order $m$. 
Suppose that $D = \sum_{i = 0}^{m-1} d_i g^i \in \N [ C_m ]$ is a minimal v-sum. We define the {\em reduced exponent} $k$ of $D$ as the reduced exponent of the vanishing sum $\tau(D) = \sum_{i =0}^{m-1} d_i \tau(g)^i$.  
\end{definition}

Note that the reduced exponent defined above does not depend on the choice of the character $\tau$. 

\begin{lemma}\label{lem-dmini}
If $D \in \N[C_m]$ is a minimal v-sum with reduced exponent $k$, then $D = D' h$ for some $D' \in \N[C_k]$ and $h \in C_m$.  
\end{lemma}

\begin{proof}
  Write $D = \sum_{i \in S} d_i g^i$ and $\tau(D) = \sum_{i \in S} d_i \tau(g^i)$ where $S\subseteq \{0,\ldots, m-1\}$. Since $k$ is the reduced exponent of $D$, by Definition~\ref{def-grexp}, the reduced exponent of $\tau(D)$ is also $k$. Thus, there exists a $j$ such that $(\tau(g^i) \tau(g^{-j}))^k = 1$ for all $i \in S$. It then follows that $Dg^{-j} \in \N[C_k]$. The proof is then completed. 
\end{proof}

In view of Proposition~\ref{prop-rr}, we derive the following result. 

\begin{corollary}\label{coro-rr}
  Suppose that $D = \sum_{i = 0}^{m-1} a_i g^i \in \N[C_m]$ is a minimal v-sum with reduced exponent $k$. Then $k$ is square free and 
  $$
  || D || \ge 2 + \sum_{p \in \cP(k)} (p-2).
  $$
\end{corollary}

%Suppose that $ X = \sum_{i \in S} a_i \mu_i = 0$. For a subset $I$ of $S$, we first define $X_I = \sum_{i \in I} a_i \mu_i$. Define
%$$
%\mathcal{I} := \{ I \subseteq S: X_I = 0 \}.
%$$
%It is clear that if $I$ is a minimal element in the set $\mathcal{I}$, then $X_I$ is an n-minimal relation. Inductively, we are able to pick minimal elements $I_1, I_2, \ldots, I_t$ in $\mathcal{I}$ such that $I_i$'s are pairwise disjoint and $X = \sum_{i=1}^t X_{I_i}$. In this way, to deal with the group element $D \in \N[C_m]$ with $\tau(D) = 0$ for a character $\tau$ of order $m$, we record the following. 

To deal with a v-sum $D \in \N [C_m]$ which is not minimal, we first decompose it into sum of minimal v-sums. It is straightforward to prove the following. 

\begin{lemma}\label{lem-smini}
Let $D \in \N[C_m]$ be a v-sum. Then $D$ can be written as the form $D = \sum D_i$, where $D_i$'s are minimal v-sums in $\N[C_m]$.  
\end{lemma}

We aim to find a lower bound of $||D||$ when $D$ is a v-sum. To do so, we need to extend the notion of reduced exponent and then apply Corollary~\ref{coro-rr}. Suppose that $D = \sum_{i = 1}^t D_i$ and $k_i$ is the reduced exponent of $D_i$ for each $i$. We may then define the exponent of $D$ to be $\lcm (k_1, \ldots, k_t)$. However, we note that such a decomposition is not necessarily unique. For example, if $m = 10$ and $h$ is a generator of $C_{10}$, then we have
\begin{eqnarray*}
   D = \sum_{i=1}^9 h^i & = & (1+h^5) + (1+h^5)h + (1+h^5)h^2 + (1+h^5) h^3 + (1+h^5) h^4 \quad \textrm{ and } \\
   D = \sum_{i=1}^9 h^i & = & (1+h^2+h^4+h^6+h^8) + (1+h^2+h^4+h^6+h^8)h.
 \end{eqnarray*}
Note that $(1+h^5) h^i$ and $(1 + h^2 + h^4 + h^6 + h^8) h^j$ are both minimal v-sums. If we use the notion of $\lcm$ of each decomposition, we will then get $2$ and $5$ as the reduced exponents, respectively. Thus, we need to modify the earlier definition of exponent as follows. 

\begin{definition}\label{def-cexp}
  Suppose that $ D = \sum_{i = 0}^{m-1} d_i g^i$ is a v-sum in $\N[C_m]$. We define the {\em c-exponent} of $D$ to be the smallest $k$ such that there exist $t$ minimal v-sums $D_1, \ldots, D_t$ in $\N[C_m]$ with $D = \sum_{i=1}^t D_i$ and $k = \lcm (k_1, \ldots, k_t)$, where $k_i$ is the reduced exponent of $D_i$ for $i = 1, \ldots, t$. 
\end{definition} 

Note that in the example above, the c-exponent of $D$ is $2$. 

%\begin{definition}\label{def-cexp}
  %Let $g$ be a generator of $C_m$ and $\tau$ be a character of order $m$. Suppose
  %$S\subset \{0,\ldots, m-1\}$ and $D = \sum_{i \in S} d_i g^i \in \N[C_m]$ and $\tau(D) = 0$. The {\em c-exponent} of $D$ is defined to be the c-exponent of the vanishing sum $\tau(D) = \sum_{i \in S} d_i \tau(g)^i$. 
%\end{definition}

%We next extend {\bf c-exponent} of a vanishing sum which needs not to be minimal. 
%We now present the following several lemmas, which will be used later.

\begin{lemma}\label{lem-exp}
%Let $g$ be a generator of $C_m$ and $\tau$ be a character of order $m$. 
Suppose that $D = \sum_{i =0}^{m-1} d_i g^i  \in \N[C_m]$ is a v-sum with c-exponent $k$. Write $m = \prod_{i = 1}^s p_i^{\alpha_i}$ and $k = \prod_{i=1}^t p_i$. Note that $t \leq s$ and $p_i$'s are distinct primes. Then we have the followings:   
\begin{itemize}
  \item[(a)] $||D|| \geq 2 + \sum_{i = 1}^t (p_i - 2)$;
  \item[(b)] $D = \sum_{i = 1}^t P_i E_i$, where $P_i$ is the subgroup of order $p_i$ and $E_i \in \Z[C_m]$ for all $i$; 
  \item[(c)] Suppose that $\prod_{i = 1}^t p_i^{\alpha_i} | d$ and $d | m$. If $\phi: \Z[C_m] \rightarrow \Z[C_d]$ is the natural projection, then $\chi(\phi(D)) = 0$ whenever $\ord(\chi) = d$.
\end{itemize}
\end{lemma}

\begin{proof}
  By Lemma~\ref{lem-smini}, we may assume that $D = \sum_{i=1}^t D_i$ such that each $D_i$ is a minimal v-sum. Hence, by Corollary~\ref{coro-rr}, we have 
\begin{eqnarray*}
  \lefteqn{||D|| = \sum_{i=1}^t |D_i| } \\
  & &  \geq \sum_{i=1}^t [2+\sum_{q\in {\cP}(k_i)} (q-2)] \\
  & & \geq 2+\sum_{q\in {\cP}(k)} (q-2) \\ 
  & & = 2+\sum_{i=1}^t (p_i-2) ,
\end{eqnarray*}
because ${\cP}(k)=\bigcup_{i=1}^t {\cP}(k_i)$. 

%Let $D_i=\sum_{j\in I_i} a_j g^j$. Then $\tau(D_i)=X_{I_i}$. As the reduced exponent of $X_{I_i}$ is $k_i$, 
By Lemma~\ref{lem-dmini}, $D_i=E_i g_i$ where $E_i\in \N[C_{k_i}]$ and $g_i \in C_m$. Clearly, $\tau(E_i) = 0$. Therefore, from~\cite[Theorem 2.2]{LL00}, it follows that $E_i=\sum_{q\in {\cP}(k_i)} Q_qF_q$, where $Q_q$ is the subgroup of order $q$ and $F_q\in\Z[C_{k_i}]$. Since $D=\sum D_i$, $D$ is of the desired form. 

Finally, note that if $\phi$ and $\chi$ are defined as in (c), then $\chi (\phi(D)) = 0$ as   $\chi(\phi(P_i))=\chi(P_i)=0$ for $i=1,\ldots, t$. 
\end{proof}
 
%Note that $X = \tau(D)=\sum_{i\in S}  a_i\tau(g)^i=0$. As before, we may write $X=\sum_{i=1}^t X_{I_i}$ where $I_i$'s are pairwise disjoint with 
%$\bigcup_{i=1}^t I_i = S$, and each $X_{I_i}$ is a minimal relation with exponent $k_i$. 
%Hence, by Proposition~\ref{prop-rr}, we have 

%\begin{lemma}\label{lem-decom}
%Let $D\in \N[C_m]$. Suppose that $\tau(D)=0$ for a character $\tau$ of order $m$. Then $D=\sum D_i$ with each $D_i\in \N[C_m]$ and n-minimal. 
%\end{lemma}

Next, we record a very useful result from~\cite[Theorem 4.8, Proposition 6.2]{LL00}.

\begin{proposition}\cite{LL00}\label{prop-nmini}
Let $D\in \N[C_m]$ be a minimal v-sum with c-exponent $k$. Then we have the followings:
\begin{itemize}
\item [(a)] If $k=p$ is prime and $P$ is the subgroup of order $p$, then $D=Ph$ for some $h\in C_m$.
\item [(b)] If $k=\prod_{i=1}^t p_i$ with $t\geq 2$ and $p_1<p_2<\cdots <p_t$ are primes,  then 
$t\geq 3$ and 
\[ ||D|| \geq  (p_1-1)(p_2-1)+(p_3-1).\]
Moreover, equality holds only if $D=(P_1^*P_2^*+P_3^*)h$ for some $h \in C_m$. Here $P_i^*=P_i-\{e\}$, and $P_i$ is the subgroup of order $p_i$. 
\end{itemize}
\end{proposition}

\begin{remark}\label{rmk-k}
  It follows from Proposition~\ref{prop-nmini} that either $k$ is a prime or $k$ has at least three prime factors. 
\end{remark}

\section{New nonexistence results of $(m,n)$-GBFs}\label{sec-main} 

In this section, we derive some new necessary conditions on $(m,n)$-GBFs, and then give new nonexistence results accordingly. First we fix the following notation. As before, we assume that $g$ is the generator of $C_m$, and note that Remark~\ref{rmk-not} holds for any GBF $f$. To avoid confusion, we set $g^0$ as the identity element of $C_m$. 

The following result is very important, in the sense that it allows to eliminate all prime factors of $m$ greater than $2^n$ when deriving nonexistence results.

\begin{proposition}\label{prop-ggbf}
  Suppose that $f$ is an $(m,n)$-GBF and $m=\prod_{i=1}^s p_i^{\alpha_i}$ where $p_i$'s are distinct primes. Let  $k_x$ be the c-exponent of $E_x$ (as defined by (\ref{eqn-dfex})) for each $1_G\neq x\in G$. 
 Set 
 \[ I=\{1 \leq i \leq s: p_i \nmid k_x \ \forall x\in G\} \mbox{ and } \overline{m} = \prod_{i\notin I} p_i^{\alpha_i}.\]
 Then there exists an $(\overline{m}, n)$-GBF. In particular, if $p_i|m$ and $p_i > 2^n$, then there exists an $(m/p_i, n)$-GBF.
  
  %If $q$ is a prime such that $q | m$ and $q$ does not divide the reduced exponent of $E_x$ as defined in (\ref{eqn-dfex}) for any $x \in G = C_2^n$, then there exists an $(m/q, n)$-GBF. In particular, if $q | m$ and $q > 2^n$, then there exists an $(m/q, n)$-GBF. 
\end{proposition}

\begin{proof}
 By induction, it suffices to show that if $p_i\in I$, then there exists an $(m/p_i, n)$-GBF. 
 Let $\eta: \Z[\langle g \rangle] \rightarrow \Z[\langle g^{p_i} \rangle]$ be the natural projection, it then follows that
$$
\eta(D_f) \eta(D_f)^{(-1)} = 2^n + \sum_{1_G \ne x \in G} \eta(E_x) x . 
$$
Recall that $E_x$ is a v-sum. By assumption $p_i$ does not divide $k_x$ for all $1_G \neq x \in G$. It follows from Lemma~\ref{lem-exp} (c) that $\tau(\eta(E_x)) = 0$ if $\tau$ is a character of order $m/p_i$. Therefore, $\tau(\eta(D_f))$ gives rise to an $(m/p_i,n)$-GBF.  

The last statement is now clear as if $p_i>2^n$, then by Lemma~\ref{lem-exp} (a), $p_i$ does not divide $k_x$ for any $1_G\neq x\in G$. 

\end{proof}

We record the following result which will be used from time to time later.

\begin{lemma}\label{lem-ab}
Suppose that $f$ is an $(m,n)$-GBF, and $p,q$ are distinct primes that both divide $m$. Then 
there exist $y\neq 1_G$ and $h\in supp(E_y)$ such that $pq|\circ(h)$.  
\end{lemma}
\begin{proof} 
As $C_m=\langle \{ g^{f(x)}: x\in G\}\rangle$, there exist $u,v\in G$ such that
$p|\circ(g^{f(u)})$ and $q|\circ(g^{f(v)})$. 
Since $g^{f(1_G)}=g^0 \in C_m$, we know that $ g^{f(u)}\in supp(E_u)$ and $ g^{f(v)}\in supp(E_v)$. 
We are done if $q|\circ(g^{f(u)})$ or $p|\circ(g^{f(v)})$. Otherwise, 
$ug^{f(u)}(vg^{-f(v)})\in supp(E_{uv})$ and then clearly $pq|\circ (g^{f(u)-f(v)})$. The proof is completed.  
\end{proof}

%\begin{lemma}\label{lem-ab}
%Suppose that $A, B \in \Z[H]$, where $H$ is an abelian group. If $||AB|| = ||A|| \cdot ||B||$, then
%\begin{itemize}
  %\item[(a)] $Supp(AB) = \{gh: \textrm{ $g \in Supp(A)$, $h \in Supp(B)$}\}$;
  %\item[(b)] $AB \in \Z$ if and only if $A = ah$ and $B = bh^{-1}$ for some $a, b \in Z$ and $h \in H$.
%\end{itemize}
%\end{lemma}

%\begin{proof}
%Note that the assumption $||AB||=||A||\cdot ||B||$ means that there is no cancellation when $AB$ is expanded. It is then clear that (a) holds.  (b) is now clear as both $|Supp(A)|=|Supp(B)|=1$.
%\end{proof}

Before we proceed, we need a technical result.

\begin{lemma}\label{lem-q1q2q3}
Let  $q_1,q_2,q_3$ be primes that divide $m$ and $Q_1,Q_2,Q_3$ be subgroups of order $q_1,q_2,q_3$, respectively. Suppose that $4\nmid m$ and $\sum_{i=1}^t Q_ih_i=\sum_{i=1}^t Q_ih_i^{-1}$ for some   $h_1,h_2,h_t\in C_m$ with $t \ge 2$. 
\begin{itemize}
\item [(a)] If $q_1\neq q_2$ and $t=2$, then we may assume $h_i^{-1}=h_i$ for $i=1,2$. 
\item [(b)] If $q_1\neq q_2=q_3$ and $t=3$, then we may assume $Q_2h_2+Q_2h_3=Q_2(h_2+h_2^{-1})$ and $h_1=h_1^{-1}$.
\item [(c)] If all $q_i$'s are distinct, then we may assume $h_i=h_i^{-1}$ for all $i$.
\end{itemize}
\end{lemma}

\begin{proof} 
By assumption, we  have
\[ Q_1(h_1-h_1^{-1})=\sum_{i=2}^t Q_i(h_i^{-1}-h_i).\]

Suppose that $q_1^{\beta_1}|| m$. Let $\phi:\Z[C_m]\rightarrow \C[C_m]$ be a ring homomorphism that fixes $g^{m/q_1^{\beta_1}}$ and sends $g^{q_1^{\beta_1}}$ to an $m/q_1^{\beta_1}$-primitive root of unity. Then,  we have $\phi (Q_i(h_i-h_i^{-1}))=0$ for $i=2, \ldots, t$, which implies that $\phi(Q_1h_1-Q_1h_1^{-1})=0$. Write $h_1=g_1 h'$ with $g_1\in \langle g^{m/q_1^{\beta_1} }\rangle$ and $p_1\nmid \circ(h')$. 
Then, we have $ Q_1g_1\phi(h')=Q_1g_1^{-1}\phi(h'^{-1})$. 
Hence  $g_1^2\in Q_1$ and $\phi(h')=\phi(h'^{-1})$. 
If $q_1$ is odd, then $g_1=g^0$. If $q_1=2$, then as $4\nmid m$, $g_1$ can be taken as $g^0$ as well.
In both cases, we may assume $g_1 = g^0$. 
It follows that $\phi(h')^2=1$. As $\phi$ is of order $m/q_1^{\beta_1}$, $h'^2=g^0$. Therefore, $g_1h'=(g_1h')^{-1}$. 
Furthermore, we have 
\begin{equation}
\sum_{i=2}^t Q_i(h_i^{-1}-h_i)=0.
\end{equation}

Now (a) follows easily by applying the same argument on $Q_2$. 

If $t=3$ and $q_2=q_3$, we then obtain $Q_2(h_2+h_3)=Q_2(h_2^{-1}+h_3^{-1})$. If $Q_2h_2=Q_2h_2^{-1}$, then we must have $Q_2h_3=Q_2h_3^{-1}$. Then, $h_2^{2}\in Q_2$ and $h_3^2\in Q_2$. 
Using a similar argument as before, we may assume that $h_2=h_3=g^0$. 
If  $Q_2h_2=Q_2h_3^{-1}$, then clearly, we may take $h_3=h_2^{-1}$ and we are done. 

To obtain (c), we set $t = 3$. We then get our desired results by applying part (a) to Equation (7).

The proof is then completed. 
\end{proof}

Now we are able to give the following necessary conditions on the existence of $(m,n)$ GBFs, where $m$ is odd.   

\begin{theorem}\label{thm-main1}
Suppose that $m = \prod_{i = 1}^s p_i^{\alpha_i}$, where $3 \leq p_1 < p_2 < \cdots < p_s$ are odd primes and $\alpha_i$'s are all positive integers. If an $(m,n)$-GBF exists, then $s \geq 2$ and $3p_1 + p_2 \leq 2^n$.
\end{theorem}
 
%Suppose that $m = \prod_{i = 1}^s p_i^{\alpha_i}$, where $3 \leq p_1 < p_2 < \cdots < p_s$ are odd primes and $\alpha_i$'s are all positive integers. If an $(m,n)$-GBF exists, then we have the following:  
%\begin{itemize}
  %\item[(i)] $s \geq 2$ and $3p_1 + p_2 \leq 2^n$; 
  %\item[(ii)] if $p_{r+1}$ is the smallest prime such that $p_1 + p_{r+1} > 2^n$, then a $(\prod_{i = 1}^r p_i^{\alpha_i}, n)$-GBF exists. 
%\end{itemize}

\begin{proof}
Recall that if $1_G\neq x \in G$ and $\chi$ is a character of order $m$,  then $\chi(E_x) = 0$. 
If $s=1$, then by Lemma~\ref{lem-exp} (b), $E_x=P_1 W$ where $P_1$ is a subgroup of order $p_1$ and $W\subseteq C_m$. In other words, $2^n=||E_x||=p_1||W||$. This is impossible as $p_1\neq 2$. 

Next, we assume that $s\geq 2$.  As $E_x \in \N[C_m]$, we may write $E_x=\sum D_j$ such that all $D_j$'s are minimal v-sums. Let $k_j$ be the reduced exponent of $D_j$. If $|{\cP}(k_j)|\geq 4$, then by Corollary~\ref{coro-rr}, we have $||D_j||\geq 2+ \sum_{i=1}^4 (p_i-2) \geq 3p_1+p_2$. Thus, we may assume that $|{\cP}(k_j)|\leq 3$. But by Proposition~\ref{prop-nmini},  $|{\cP}(k_j)|=1$ or $3$. In case that $|{\cP}(k_j)|=3$, $||D_j||\geq q_1(q_2-1)+q_3-q_2\geq p_1(p_2-1)+p_3-p_2$. If $p_1\geq 5$, then clearly, $p_1(p_2-1)+p_3-p_2\geq 3p_1+p_2$. If $p_1=3$, it then follows that  
$$
p_1(p_2-1)+p_3-p_2\geq 2p_2+(p_3-2)\geq p_2+(5+7)\geq 3p_1+p_2
$$
as $p_2\geq 5$ and $p_3\geq 7$. 

It remains to consider the case $|{\cP}(k_j)|=1$, i.e., $D_j= Q_j h_i$ where $h_i\in C_m$ and $Q_j$ is a subgroup of order $q_j$. Note that $q_j$'s need not be distinct. Therefore, $E_x=\sum_{j=1}^t Q_j h_j$. If all $Q_j$'s are the same, then $E_x=Q_1Y$ for some $Y\in \Z[C_m]$. This is impossible as $q_1\nmid 2^n$. In particular, it follows that $t\geq 2$ and we may assume  $Q_1\neq Q_2$ without loss of generality.  Recall that all $D_i\in \N[C_m]$. Therefore, 
\[ 2^n=||E_x|| \geq q_1+q_2+(t-2)p_1.\]
Hence, we are done if $t \geq 4$.  

We first study the case $t = 3$. As $q_1\neq q_2$, we may assume $q_1\neq q_3$ as well. 
Since $E_x^{(-1)}=E_x$ and $m$ is odd, we may then assume $h_1=g^0$. 
Moreover, if $Q_2=Q_3$, then $Q_2h_2+Q_2h_3=Q_2(h_2+h_2^{-1})$. Whereas if $Q_2\neq Q_3$, then $h_2=h_3=1_G$ as $m$ is odd. Therefore, the coefficient of $g^0$ is either $1$ or $3$ in both cases. 
This contradicts Lemma~\ref{lem-ex} (a). 

Thus, we may assume $t = 2$ for all $x \ne 1_G$. Moreover, as $m$ is odd, $E_x$ is of the form $Q_1+Q_2$. In particular, each non-identity element in $supp(E_x)$ is of prime order. This contradicts Lemma~\ref{lem-ab}.   

The proof is then completed. 
\end{proof}

The theorem above provides an alternative proof of~\cite[Corollary 2]{LFF17}, from which we can have an improved result on the case $s = 2$.

\begin{corollary}\label{coro-main1}
  Suppose that $m = \prod_{i = 1}^s p_i^{\alpha_i}$, where $p_1 < p_2 < \cdots < p_s$ are odd primes and $\alpha_i$'s are all positive integers.
\begin{itemize}
  \item[(a)] There is no $(m,n)$-GBF when $s = 1$. 
  \item[(b)] There is no $(m,n)$-GBF if $s \geq 2$ and $3p_1 + p_2 > 2^n$.
  \item[(c)] There is no $(m,n)$-GBF if there is no $(\prod_{i=1}^r p_i^{\alpha_i}, n)$-GBF where $p_{r+1}$ is the smallest prime such that $p_1 + p_{r+1} > 2^n$. 
\end{itemize} 
\end{corollary}

\begin{proof}
  (a) and (b) follow directly from Theorem~\ref{thm-main1}. As for (c), it suffices to show that if $t\geq r+1$, then $p_t$ does not divide the c-exponent of $E_x$ for any $x\neq 1_G$. We follow the notation used in the proof of Theorem~\ref{thm-main1}. We write $E_x=\sum D_j$ such that all $D_j$'s are minimal v-sums. Again, we denote by $k_j$ the reduced exponent of $D_j$. Suppose that $p_t|k_1$. If $k_1 = p_t$, then $E_x\neq D_1$ as otherwise $p_t|2^n$. Therefore, $||E_x||\geq ||D_1||+||D_2||\geq p_t+p_1>2^n$. On the other hand, if $k_1\neq p_t$, then as shown before, $k_1$ is a product of at least three primes. Hence, $||D_1||\geq p_t+p_1>2^n$, which is impossible.
\end{proof}

\begin{remark}\label{rmk-main1}
For $s = 2$, our result is stronger than ~\cite[Corollary 2]{LFF17}. 
\end{remark}

Now we consider the case when $m = 2m'$ with $m'$ odd. If $f$ is a $(2m', n)$ GBF, then we define 
$$
G_f := \{ x \in G : f(x) \textrm{ odd} \}.
$$
Note that a $(2m', n)$ GBF is trivially an $(m',n)$ GBF if $G_f = \emptyset$ or $G$. Multiply $f$ by $-1$ if necessary, we may always assume $|G_f|\leq |G|/2$.  Note that $G_f^{(-1)}=G_f$ as $G$ is $2$-elementary. Apply a homomorphism $\psi: \Z[G\cdot C_{m}]$ such that $\psi$ fixes every element in $G$ and maps the generator $g$ of $C_m$ to $-1$, then we have 
\begin{eqnarray*} 
  \lefteqn{\psi(D_f)\psi(D_f^{(-1)})} \\
  & = &  (G-2G_f)(G-2G_f^{(-1)}) \\
  & = & (|G|-4|G_f|)G+4G_f^2 \\
  & = & 2^n+\sum_{1_G \neq x\in G} \psi(E_x)x.
\end{eqnarray*}
 Write  
 \begin{equation}\label{eqn-bx}
 G_f^2 = |G_f|+ 2 \sum_{ 1_G \neq x\in G} b_x x. 
 \end{equation}
We denote $\psi(E_x)$ by $a_x$.  It then follows that for $x \ne 1_G$,  
\begin{equation}\label{eqn-ax}
a_x=|G|-4|G_f|+8b_x .
\end{equation}

The following is a consequence of~\cite[Theorem 1]{Mann65}.

\begin{lemma}\label{lem-mann}
If $n$ is odd, then $G_f$ is a difference set in $G$ if and only if $G_f=\{ 1_G \}$.
\end{lemma}

We now give the following nonexistence results on $(2m', n)$ GBFs, which are weaker than those in Theorem~\ref{thm-main1}.

\begin{theorem}\label{thm-main2}
  Let $n$ be odd and $m = 2 p^\alpha $, where $\alpha$ is a positive integer. Suppose that an $(m,n)$-GBF exists. Then $p < 2^{n-3}$ unless $p = 2^{n-2} - 1$ is a Mersenne prime. In particular, if $n \leq 3$, there is no $(m,n)$-GBF if $m = p^\alpha$ or $m = 2p^\alpha$.  
\end{theorem}
%If $s \ge 2$ and $r$ is the least integer such that $p_{r+1}+p_1>2^n+2$, then a $(2\prod_{i=1}^r p_i^{\alpha_i}, n)$-GBF exists. In particular, we have $p_1<2^{n-2}$ or $p_1+p_2+2\leq 2^n$.

\begin{proof}
  Let $P_2$ be the subgroup of order $2$ and $P$ be a subgroup of order $p$. For any $x \ne 1_G$, we conclude from Lemma~\ref{lem-exp} (b) that $E_x = P_2 Y_x + P Z_x$ for some $Y_x, Z_x \in \N[C_m]$. Note that $\psi(E_x) \ne 0$ for some $x \ne 1_G$. Otherwise, the c-exponent of all $E_x$ is $2$ and by Proposition~\ref{prop-ggbf}, there exists a $(2,3)$-GBF, which is impossible. Hence, $a_x = \psi(E_x) \ne 0$ for some $x \ne 1_G$. Therefore, we have $\psi(P) | \psi(E_x)$, i.e., $p | a_x$. Note that in view of Equation (\ref{eqn-ax}), $4p | a_x$ if $|G_f|$ is odd and $8p |a_x$ if $|G_f|$ is even. We are done if $8p | a_x$ as $|a_x| < 2^n$. We may therefore assume that $|G_f|$ is odd.  
 
Suppose that $G_f = \{ 1_G \}$. Then, $a_x = 2^n-4$ if $x \ne 1_G$. Hence, $4p | a_x$. It follows that $p_1 < 2^{n-3}$ unless $4p=2^n-4$ which implies that $p = 2^{n-2}-1$ is a Mersenne prime. 

%Observe that if $|G_f|$ is even, then we have $8|a_x$. Hence, $8p_1<2^n$ and we are done in that case. 

Suppose that $G_f \neq \{1_G\}$. As $G_f$ is not a difference set, there exist two elements $x\neq 1_G$ and $x'\neq 1_G$ such that $b_x> b_{x'}\geq 0$. Since $p | a_x$ and $p | a_{x'}$, it follows that $p | (b_x-b_{x'})$ and $b_x - b_{x'} = tp$ for some positive integer $t$. To get our desired result, we need to find a bound on $b_x-b_{x'}$. Note that in view of Equation (\ref{eqn-bx}), $b_x \leq |G_f|/2 \leq |G|/4$. Hence, we get our desired result if $t \geq 2$. Thus, we
may assume that $t=1$, i.e., $b_x = p + b_{x'}$.

Suppose that $G = \langle x\rangle \cdot G'$, where $G'$ is a subgroup of order $2^{n-1}$ in $G$. As the coefficient of $x$ in $G_f^2$ is $2b_x$, there are $2b_x = 2p + 2b_{x'}$ pairs $(u,v)$ of elements in $G_f\times G_f$ such that $uv=x$. Therefore, there exists a set $Y \subseteq G'\cap G_f$ such that $Y\cup (Yx)\subseteq G_f$ with $|Y| = p + b_{x'}$. Write $G_f = (Y \cup Z_1) \cup (Yx \cup Z_2x)$ such that 
\[ 
Z_1\subseteq G', Z_2\subseteq G', Y\cap Z_1=\emptyset \mbox{ and } Y \cap Z_2=\emptyset.
\]
Since $b_x = |Y|$, it follows that $Z_1\cap Z_2=\emptyset$.  Moreover, we have
\[ 
G_f^2 = [2Y^2+2Y(Z_1+Z_2)+Z_1^2+Z_2^2]+[2Y^2+2Y(Z_1+Z_2)+2Z_1Z_2]x.
\]
Note that the support of $[2Y^2+2Y(Z_1+Z_2)+Z_1^2+Z_2^2]$ is in $G'$ and 
the support of $[2Y^2+2Y(Z_1+Z_2)+2Z_1Z_2]x$ is in $G'x$. 
We now consider the coefficients of the following group elements
\[Z=[2Y^2+2Y(Z_1+Z_2)+Z_1^2+Z_2^2]-[2Y^2+2Y(Z_1+Z_2)+2Z_1Z_2]=(Z_1-Z_2)^2.\]
For any $1_{G}\neq v\in G'$, the coefficient of $v$ in $Z$ is equal to $2(b_v-b_{vx})$. Clearly, the absolute value of the coefficient of $v$ in $Z$ is less than $|Z_1|+|Z_2|$ as $Z_1$ and $Z_2$ are disjoint. Thus, if there exists $v \neq 1_G$ in $G'$ such that $b_v-b_{vx}$  is nonzero, then $p | (b_v - b_{vx})$ and we obtain
\[ 
2p \leq 2|b_v-b_{vx}| \leq |Z_1|+|Z_2| \leq (|G_f|-2b_x) \leq |G_f| - 2p.
\]
Hence, we get $4p \leq |G_f| \leq |G|/2$ and $p \leq 2^{n-3}$. Thus, it remains to deal with the case $(Z_1-Z_2)^2=|Z_1|+|Z_2|$. 

If both $Z_1 = Z_2 = \emptyset$, then $2b_x = |G_f|$. Hence, $|G_f|$ is even and as remarked earlier, we are done in this case. Note that as $G = C_2^n$, all character values of $Z_1 - Z_2$ are integers. Thus, $|Z_1|+|Z_2|$ is a square. Since $Z_1 \cap Z_2 = \emptyset$, all nonzero coefficients of $Z_1-Z_2$ is $\pm 1$. On the other hand, if $q$ is an odd prime divisor or $|Z_1| + |Z_2|$, then $q$ divides the all nonzero coefficients of $Z_1 - Z_2$ by applying
Fourier inversion formula. This is impossible. It follows that $|Z_1|+|Z_2| = 2^t$. Again, we are done if $t \geq 1$ as then $|G_f|=2b_x+|Z_1|+|Z_2|$ is even.  Hence, we may assume that $t=0$, i.e., $|Z_1|+|Z_2|=1$. Note that the coefficient of $1_G$ in $[2Y^2+2Y(Z_1+Z_2)+Z_1^2+Z_2^2]$ is $|G_f|$ and the coefficient of $1_G$ in $[2Y^2+2Y(Z_1+Z_2)+2Z_1Z_2]$ is the same as the coefficient of $x$ in $G_f^2$. As $Z=1$, it follows that $2b_x=|G_f|-1$. Hence, $a_x=|G|-4|G_f|+4(|G_f|-1)=2^n-4$. Recall that $4p | a_x$. Hence either $p=2^{n-2}-1$ or $p < 2^{n-3}$.

The proof is then completed. 
\end{proof}

\begin{corollary}\label{coro-main2}
  Let $n$ be odd and $m = 2 \prod_{i =1}^s p_i^{\alpha_i}$, where $p_1 < p_2 < \cdots < p_s$ are odd primes and $\alpha_i$'s are all positive integers.  
  \begin{itemize}
    \item[(a)] If $s = 1$, then there is no $(m,n)$-GBF if one of the following conditions is satisfied:
      \begin{itemize}
	\item[(i)] $p_1 > 2^{n-2}$;
	\item[(ii)] $p_1$ is not a Mersenne prime and $p_1 > 2^{n-3}$;
	\item[(iii)] $p_1 \equiv 3, \ 5 \pmod{8}$.  
  \end{itemize} 
\item[(b)] If $s \ge 2$, and $r$ is the least integer such that $p_{r+1} + p_1 > 2^n + 2$, then there is no $(m,n)$-GBF if there is no $(2\prod_{i=1}^r p_i^{\alpha_i}, n)$-GBF. In particular, there is no $(m,n)$-GBF if $p_1 > 2^{n-2}$ and $p_1 + p_2 > 2^n +2$. 
  \end{itemize}
\end{corollary}

\begin{proof}
  It is easily seen that (i) and (ii) of (a) directly follow from Theorem~\ref{thm-main2}. If (iii) holds, it is known that no $(2 p_1^{\alpha_1}, n)$-GBF exists.  
  
  To prove (b), it is sufficient to show that for $i\geq r + 1$,  $p_i$ does not divide the c-exponent of any $E_x$ for $x \ne 1_G$.  As before, we wirte $E_x = \sum D_j$ and $k_j$ the reduced exponent of $D_j$. We may assume that $p_i$ divides $k_1$. If $k_1$ consists of at least three prime factors, then $|| D_i || \ge 2 + (p_1 -2) + (p_i - 2)$. Thus, $2^n \ge p_1 + p_i - 2 \ge p_1 + p_{r+1} - 2 > 2^n$. This is impossible. Therefore, we have $k_1 = p_i$.  

Otherwise, we assume that $p_i$ divides the reduced exponent $k_x$ of $\tau(E_x)$. If $k_x=p_i$, it follows from the argument in (a) that $4p_i \leq 2^n$. This is impossible as $2^n < p_1 + p_i < 4p_i$. Therefore, $p_j|k_x$ for some $j\neq i$. But then by Proposition~\ref{prop-nmini}, $2^n \geq p_j+p_i-2 > p_{r+1}+p_1-2$. This is impossible. 
\end{proof}

\begin{remark}\label{rmk-main2}
  When compared with~\cite[Theorem 2]{LFF17}, our result in Corollary~\ref{coro-main2} is stronger in all cases quoted in Table 2~\cite{LFF17} therein except for the case that $p = 191$. 
\end{remark}

\section{Nonexistence results for $n=3$}\label{sec-n3}

In this section, we show that there in no $(m, 3)$-GBF for all $m$ odd or $m\equiv 2 \pmod{4}$.  By Proposition~\ref{prop-ggbf}, we may assume that all prime factors of $m$ are less than or equal to $7$. According to Corollary~\ref{coro-main1}, we conclude that there is no $(m,3)$-GBF if $m$ is odd. Therefore, we may write $m=2\cdot 3^a 5^b 7^c$. For convenience, we fix the following notation. Let $g_2,g_3,g_5,g_7$ be elements of order $2,3,5,7$, respectively. Let $P_2, P_3,P_5,P_7$ be subgroups of order $2,3,5$ and $7$, respectively. 

We assume that $f$ is an $(m,3)$-GBF. We first determine what $E_x$ is if $x\neq 1_G$. As seen before, $\tau(E_x)=0$ for any character of order $m$. Recall that $\cP(k)$ denotes the set of all prime factors of the integer $k$.  

\begin{lemma}\label{lem-pki}
For any $x\neq 1_G$, write $E_x=\sum D_i$ where each $D_i$ is a minimal v-sum with reduced exponent 
$k_i$. Then ${\cP}(k_i) = \{2\}, \{3\}, \{5\}, \{7\}$ or  $\{2,3,5\}$ or $\{2,3,7\}$. Moreover, 
\begin{itemize}
  \item [(a)] If ${\cP}(k_i)=\{j\}$ for some $j\in \{2,3,5,7\}$, then $D_i = P_j h_j$ for some $h_j \in C_{30}$.   
\item [(b)] If ${\cP}(k_i)=\{2,3,7\}$, then $E_x = g_2^{\alpha} (P_7^* + g_2 P_3^*)$ for some integer $\alpha$. 
\end{itemize}
\end{lemma}

\begin{proof} 
  Let $k_x$ be the reduced exponent of $E_x$. Note that $k_x\neq 2\cdot 3 \cdot 5 \cdot 7$, $3 \cdot 5 \cdot 7$, or $2\cdot 5 \cdot 7$ as $||E_x|| > (7-2)+(5-2)+2 > 8$.  Therefore, either $|{\cP}(k_i)|=1$,  ${\cP}(k_i)=\{2,3,7\}$ or $\{2,3,5\}$. (a) then follows from Lemma~\ref{lem-dmini}. 

For (b), note that $||D_i|| \leq 8$. Hence, by Proposition~\ref{prop-nmini} (b), $D_i = h  (P_7^* + g_2P_3^*)$ for some element $h \in C_m$. As $||E_x||=8$, $E_x = D_i$. As $E_x = E_x^{(-1)}$, we have $h=g_2^{\alpha}$ for some integer $\alpha$. 
\end{proof}

\begin{corollary}\label{coro-7kx}
If $7|k_x$, then $E_x = g_2^{\alpha}(P_7^* + g_2P_3^*)$. 
\end{corollary}

\begin{proof} 
We will follow the notation used above. By assumption, $7|k_i$ for some $i$. If $k_i=7$, then $D_i=P_7 h_i$. Since  $||E_x||=8$, it follows that $||D_j||=1$ if $j\neq i$. This is impossible as then $\tau(D_j)\neq 0$. Hence, $k_i$ is not a prime and therefore, $k_i=2 \cdot 5 \cdot 7$. By Lemma~\ref{lem-pki} (b), our desired result follows. 
\end{proof}

Let $\psi$ be as defined in Section~\ref{sec-main}. As we have seen before, $a_x=\psi(E_x)\equiv 0 \bmod 4$. With the condition $E_x=E_x^{(-1)}$, this allows us to narrow down the possibilities of $E_x$ when $7$ does not divide the c-exponent of $E_x$.

\begin{lemma}\label{lem-7kx}
If $7\nmid k_x$, then $E_x$ is in one of the forms below: 
\begin{itemize}
\item [(a)] $E_x=P_2W$ and $a_x=0$.
\item [(b)] $E_x=(P_3+P_5)g_2^{\alpha}$ and $a_x=\pm 8$.
\item [(c)]  $E_x=g_2^{\alpha}[g_2(g^0+g_5+g_5^4)(g_3+g_3^2)+(g_5^2 + g_5^3)]$ or $g_2^{\alpha}[g_2(g^0+g_5^2+g_5^3)(g_3+g_3^2)+(g_5 +g_5^4)]$ and $a_x=\pm 4$. In particular, $supp(E_x) \cap P_2 = \emptyset$. [Recall that $g^0$ is the identity of $C_{m}$. ]
\end{itemize}
\end{lemma}

\begin{proof}
  We continue with the notation used in Lemma~\ref{lem-pki}. If all $k_i$'s are prime, then in view of Lemma~\ref{lem-pki}, 
\[ 
 E_x=P_2X+P_3Y+P_5Z,
\]
 where $X,Y,Z\in \N[C_m]$. As $||E_x||=8$ and $8=2||X||+3||Y||+5||Z||$. It is clear that 
\[ 
 (||X||,||Y||,||Z||)=(4,0,0), \ (1,2,0), \ \mbox{ or } (0,1,1).
\]
 
If $(||X||,||Y||,||Z||)=(1,2,0)$, then $E_x=P_2(h_1+h_2)+P_3h_3$. In this case, $\psi(E_x)=\pm 3$. This is impossible. Next, if $(||X||,||Y||,||Z||)=(4,0,0)$, then (a) holds. If $(||X||,||Y||,||Z||)=(0,1,1)$, then $E_x=P_3h_1+P_5h_2$. By Lemma~\ref{lem-q1q2q3} (a), $h_i=g_2^{\alpha_i}$. Note that $\psi(E_x)=\pm 2\neq \pm 4$ if $\alpha_1\neq \alpha_2 \bmod 2$. Since $a_x \equiv 0 \bmod{4}$, (b) holds.

As $k_i$'s are not all prime, we may assume that $k_1$ is not a prime. Then by Lemma~\ref{lem-pki}, $k_1=2\cdot 3 \cdot 5$. But then by Proposition~\ref{prop-nmini} (b), $||D_1||\geq 6$. If $E_x\neq D_1$, then $||D_2||\leq 2$. Hence $D_2=P_2 h'$ for some $h'\in C_m$ and $||D_1||=6$. Thus, $D_1=(P_2^*P_3^*+P_5^*)h$ for some $h\in C_{30}$. Since $||E_x||=8$, $E_x=D_1+D_2$. But $\psi(E_x)=\psi(D_1+D_2)=\pm 2$. This is impossible as $4|a_x$.  Hence, $E_x$ is a minimal v-sum and
$E_x = D_1 = Dh$ for some $D\in \N[C_{30}]$. As $E_x=E_x^{(-1)}$, we have $h\in C_{30}$. So, $E_x\in \N[C_{30}]$.  We may write $E_x=\sum_{i=0}^4 A_i g_5^i$, where $A_i\in \N[C_{6}]$. Clearly, 
\[ 8=\sum_{i=0}^4 ||A_i||.\]

Let $\tau$ be a character of order $30$. If $A_i=0$ for some $i$, then $\tau(A_j)=0$ for all $j$ as $\tau(E_x)=0$. Then, $E_x$ is not a minimal v-sum unless $E_x = A_j$ for some $j$. So, $k_1 | 6$ and $k_1 \ne 30$. This is impossible. Hence, $||A_i||\geq 1$ for each $i$. 

\noindent {\bf Claim.} $||A_j|| \leq 3$ for all $j = 0, \ldots, 4$. 

Otherwise, we assume that $||A_{\ell}||\geq 3$ for some $\ell$. It then follows that $||A_j||\leq 2$ if $j\neq \ell$. Since $E_x = E_x^{(-1)}$, we have $\ell=0$.  On the other hand, if $||A_j|| = 2$ for some $j$,  then again $||A_t||\neq 2$ whenever $t \neq j$. Using the condition $E_x = E_x^{(-1)}$ again, we have $j=0$. This is impossible. Hence, all other $||A_j||=1$. Thus we conclude, $||A_{0}||=4$ and $||A_i||=1$ if $i = 1, 2, 3, 4$. Write $A_1 = h$, where $h\in C_6$. As
$\tau(A_0) = \tau(h)$, we have $\tau(A_0 + g_2h) = 0$. Note that $||A_0 + g_2 h|| = 5$. Since $\tau(A_0 + g_2h) = 0$,  we may apply a similar argument as in Lemma~\ref{lem-pki} to conclude that $A_0 + g_2h = P_2 h_1 + P_3 h_2$ for some $h_1, h_2\in C_6$. Therefore, $A_0 = P_2h_1 + h_3 + h_4$ or $A_0 = h_3 + P_3h_2$ for some $h_3, h_4\in C_6$. In either case, it contradicts the assumption that $E_x$ is a minimal v-sum. 

Hence, we conclude that $||A_j|| \leq 2$ for all $j$.  Using the assumption that $E_x = E_x^{(-1)}$ again, we then obtain two possible cases.

(i) $||A_0||=||A_1||=||A_4||=2$ and $||A_2||=||A_3||=1$ or 

(ii) $||A_0||=||A_2||=||A_3||=2$ and $||A_1||=||A_4||=1$.

It remains to show that $E_x$ is of the desired form when (i) holds. We may assume that $A_i = h_i$ for some $h_i\in C_6$ for $i=2, 3$. Since $\tau(E_x)=0$ for any character $\tau$ of order $30$, we set $h_2 = h_3 = h$. As $E_x = E_x^{(-1)}$, we see that $h=g_2^{\alpha}$. 

Note that for $i=0, 1, 4$, $||A_i + g_2h||=3$ and $\tau(A_i + g_2h)=0$. Therefore, $A_i+g_2h = P_3g_2h$ as $g_2h$ is in the support of all $A_i + g_2h$. In other words, $A_i = P_3^*(g_2h)$ for $i=0, 1, 4$. It is now clear that $E_x$ is of desired form. 
This shows that (c) holds. 
\end{proof}

\begin{theorem}\label{thm-n3}
  There is no $(m,3)$-GBF for any integer $m$ odd or $m \equiv 2 \pmod{4}$. 
\end{theorem}

\begin{proof} 
  
Recall that by earlier discussion of this section, we may assume that $m = 2 \cdot 3^a \cdot 5^b \cdot 7^c$. We first remove the case $7|m$. 

We may assume that $7$ divides the c-exponent of $E_x$ for some $x \neq 1_G$. By Lemma~\ref{lem-7kx}, we see that $E_x = h^{\alpha}(P_7^* + hP_3^*)$ and $\psi(E_x)=\pm 4$. It follows from Equation (\ref{eqn-bx})  that $a_v=\pm 4$ for any $v\neq 1_G$. Therefore, $E_v$ is of the form in Corollary~\ref{coro-7kx} or Lemma~\ref{lem-7kx} (c). That means there is no element in $supp(E_v)$ of order a multiple of $21$ for any $v$. This contradicts Lemma~\ref{lem-ab}. Thus, we may assume that $7$ does not divide the c-exponent of $E_x$ for all $x\in G$. By Proposition~\ref{prop-ggbf}, it remains to show that $(2\cdot 3^a \cdot 5^b, 3)$-GBF does not exist.

In view of Lemma~\ref{lem-7kx}, $E_x\in \N[C_{30}]$ for all $x\neq 1_G$. It follows that $supp(D_f)\subset G\cdot C_{30} h'$ for some $h'\in C_m$. After multiplying $D_f$ with $h'^{-1}$, we may assume $D_f\in \N[G\cdot C_{30}]$. Recall that we may assume that $1 \leq |G_f| \leq 4$. We may assume that $1_G \in G_f$ instead of $1_G \in G \setminus G_f$. We now discuss by cases. 

\noindent {\bf Case (1)} $|G_f| = 2$.  

As $1_G \in G_f$, we write $G_f = \{1_G, v\}$. Note that $a_x = 8$ or $0$. It follows that $a_v = 8$ and $a_x = 0$ if $x \neq 1_G, v$. By Lemma~\ref{lem-7kx}, we have $E_v = P_5 + P_3$ and $E_x = P_2 W_x$ for some $W_x\in \Z[C_{30}]$ if $x\neq 1_G, v$. 

Let $\eta:\Z[G\cdot C_{30}] \rightarrow \Z[G\cdot C_{5}]$ be a ring homomorphism such that $\eta(g_2) = -1$ and $\eta(g_3) = 1$ and $\eta(g_5) = g_5$; and $\eta(x) = 1$ for all $x \in G$.  Note that $\eta(E_x) = 0$ if $a_x = 0$ as $E_x = P_2W_x$ for some $W_x\in \N[C_{30}]$. Thus, we get 
\[ 
\eta(D_f)\eta(D_f)^{(-1)} = 11+P_5.
\]

Write $\eta(D_f)=\sum a_ig_5^i$ where $a_i\in \Z$. Observe that if we further map $g_5$ to $1$, then the resulting map is just $\psi$. Thus, we have $\sum a_i=\psi(D_f)$. Then as $|G_f| = 2$, $\psi(D_f) = \sum a_i = 8 - 2\cdot 2 = 4$. By considering the coefficient of identity of $11+P_5$, we get $\sum a_i^2=12$. Thus $|a_i| \geq 2$ for some $i$. If the maximum value of $|a_i|$ is $2$, then there must be two more $a_j$'s with $|a_j|=2$ and the rest is $0$. That is impossible as then $2$ divides $\eta(D_f)$ but $2$ does not divide $11+P_5$ in $\Z[P_5]$. 

Hence, the maximum value of $|a_i|$ is $3$. Then there are exactly three $a_j$'s with $|a_j| = 1$. Since $\sum a_i = 4$, exactly one $a_i$ is $-1$. So we may assume that $\eta(D_f) = 3 + g_5 + g_5^\beta - g_5^\gamma$ with $1 \neq \beta \neq \gamma \neq 1$. Clearly, we may assume either $\beta = 4$ or $\gamma = 4$.  

If $\beta=4$, then we may take $\gamma = 2$ or $3$. Then, the coefficient of $g_5^\gamma$ is $-2$, which is impossible. If $\gamma = 4$, then $\beta = 2$ or $3$. In that case, the coefficient of $g_5^\beta$ is $-2$, which is also impossible. Therefore, we have $|G_f| \ne 2$.

\noindent {\bf Case (2)} $|G_f| = 4$.
We may assume that $G_f=\{1_G, v_1, v_2, v_3\}$. 

\noindent {\bf Subcase (a)} $v_3 = v_1v_2$ and $G_f$ is a subgroup of order $4$. Hence, $G_f^2 = 4G_f$ and $(G-2G_f)^2 = 8G_f - 8G_fv$ for some nonzero $v\in G$. Therefore, $a_x = \pm 8$ for all $x\in G$. In view of Lemma~\ref{lem-7kx}, $E_x = g_2^{\alpha} (P_3+P_5)$ for all nonzero $x\in G$. By Lemma~\ref{lem-ab}, this is impossible as there is no element in $supp(E_v)$ which is divisible by $15$ for any $v$. 

\noindent {\bf Subcase (b)} $v_3 \ne v_1 v_2$. Let $H = \{1_G, v_1, v_2, v_1v_2\}$ be the subgroup of order $4$. Then $G_f^2 = 2G + 2 - 2v_1v_2v_3$. For convenience, we write $v=v_1v_2v_3$. Thus, $a_v=-8$ and $a_x=0$ if $x\neq 1_G$ or $v$. As $v\notin H$, there exists a ring homomorphism $\eta'$ that maps $H\cdot P_3$ to identity, and $\eta'(g_2)=\eta'(v) = -1$. Then as before $\eta'(E_x)=0$ if $a_x=0$. Hence, we obtain 
\[ 
\eta'(D_f)\eta'(D_f)^{(-1)} = 8+(-1)(-3-P_5) = 11 + P_5.
\]
Write $\eta'(D_f) = \sum a_ig_5^i$. Observe that $\sum a_i=\eta'(G-2G_f)=-4$. As shown above, there is no solution in $\Z[P_5]$.  

\noindent {\bf Case (3)} $|G_f| = 1$ or $3$. Ten $a_x = \pm 4$ for all $x \ne 1_G$ in $G$. Therefore, by Lemma~\ref{lem-7kx} (c), for any $E_x$ with  $1_G\neq x\in G$, 
\begin{eqnarray*} 
E_x & = &  g_2^{\alpha}[(g^0+g_5+g_5^4)(g_3+g_3^2)+g_2(g_5^2+g_5^3)] \mbox{ or } \\  
& & g_2^{\alpha}[(g^0+g_5^2+g_5^3)(g_3+g_3^2)+g_2(g_5+g_5^4)].
\end{eqnarray*} 
Observe that if we write $E_x = \sum_{i=0}^4 W_{xi} g_5^i$, then $||W_{x0}|| = 2$ and $P_2 \cap supp(E_x) = \emptyset$. 

%In particular, we see that 
%\begin{equation}
%\{g^0, g_2\} \cap E_x =\emptyset \ \mbox{ for all } x\neq 1_G.
%\end{equation}
 
Write $D_f = \sum_{i=0}^4 B_i g_5^i$ where $B_i\in \Z[G\cdot C_6]$ and $D_f D_f^{(-1)} = \sum_{i=0}^{4} Z_i g_5^i$ with $Z_i\in \N[G\cdot C_{6}]$. For each $i$, $B_i = A_{i0}+A_{i1}g_3 + A_{i3} g_3^2$ where $A_{ij}\in \N[G\cdot P_2]$. If $||A_{ij}||\geq 2$, i.e., $A_{ij} = x_1 h_1 + x_2 h_2 + \cdots$, where $x_1, x_2 \in G$ and $h_1, h_2 \in P_2$, then $A_{ij}A_{ij}^{(-1)} = 2 + x_1x_2 h_1 h_2 + \cdots$. Hence, $supp(E_{x_1})\cap \{g^0, g_2\}\neq \emptyset$. This
contradicts Lemma~\ref{lem-7kx} (c). Thus, $|A_{ij}|\leq 1$ and $||B_i||\leq 3$. Note that 
\[ 
||Z_0|| = 8 + \sum_{x \ne 1_G} ||W_{x0}|| = 8 + 2\times 7 = 22 = \sum_{i=0}^4 ||B_i||^2. 
\]
Observe that not all $||B_i||\leq 2$. Using the equation above, we may assume that $||B_i||=||B_j||=3$ and $||B_k||=2$ for some distinct $i,j,k$. Then we have 
\[ 
B_i = \sum_{t=0}^2 u_t g_2^{\alpha_i}g_3^t = u_0g_2{\alpha_0} (1+ u_0u_1g_2^{\alpha_1-\alpha_0}g_2g_3+ u_0u_1g_2^{\alpha_2-\alpha_0}g_3^2). 
\]

Let $\phi$ be a character on $G\cdot C_{30}$ such that $\phi(u_0u_1)=(-1)^{\alpha_1-\alpha_0}$ and $\phi(u_0u_2)=(-1)^{\alpha_2-\alpha_0}$. Note that such a $\phi$ exists as $u_0u_1\neq u_0u_2$. Then, it is clear that $\phi(B_i)=0$. Thus, $|\phi(D_f)|^2 = |\phi(B_j)\zeta_5^j+\phi(B_k)\zeta_5^k|^2 = 8$. In other words, we have 
\[  
|\phi(B_j)|^2+|\phi(B_k)|^2+ \phi(B_j)\overline{\phi(B_k)}\zeta_5^{j-k}+ 
\phi(B_k)\overline{\phi(B_j)}\zeta_5^{k-j}=8.
\]
This is impossible unless $\phi(B_j)=0$ or $\phi(B_k)=0$. But then $||B_k||=2$  and $||A_{kj}||\leq 1$ imply that $\phi(B_k)\neq 0$. Thus $\phi(B_j)=0$ and $|\phi(B_k)|^2=8$. This is impossible as $||B_k||=2$. This finish showing that $|G_f|\neq 1$ or $3$.

The proof is then completed. 
\end{proof}

%\section{Conclusions}\label{sec-con}

%In this paper, we presented some new nonexistence results on $(m,n)$ generalized bent functions, which improved recent results in~\cite{LFF17}. 

%\section*{Acknowledgement}

%\bibliographystyle{abbrv}

%\bibliography{GBF}

\end{document}